\documentclass[11pt,psamsfonts]{article}
\usepackage{amsmath}
\usepackage{amssymb}
\usepackage{amsthm}
\usepackage[all,2cell,ps]{xy}

\usepackage[english]{babel}
\usepackage[T1]{fontenc}
\usepackage[latin1]{inputenc}
\usepackage{paralist}

\newtheorem{teor}{Theorem}[section]
\newtheorem*{teo}{Theorem}

\newtheorem{lemma}[teor]{Lemma}
\newtheorem{coroll}[teor]{Corollary}

\newtheorem{pro}[teor]{Proposition}
\newtheorem*{claim*}{Claim}

\theoremstyle{definition}
\newtheorem*{rmk}{Remark}

\newtheorem{ese}[teor]{Example}

\newcommand{\ls}{\leqslant }
\newcommand{\gs}{\geqslant }
\newcommand{\call}[1]{\mathcal{#1}}
\newcommand{\got}[1]{\mathfrak{#1}}
\newcommand{\g}[1]{\ensuremath{\mathbb{#1}}}
\newcommand{\lb}{[\hspace{-0.07cm}[}
\newcommand{\rb}{]\hspace{-0.07cm}]}

\newcommand{\Tt}{\tau}

\title{  {\bf Artinian level algebras of low socle degree}
\footnote{ 2000 {\it Mathematics Subject Classification.} Primary 13H10; Secondary 13H15; 14C05.
\newline
\indent{\it Keywords and Phrases.} Artinian level local algebras, Inverse system, Isomorphism classes, Hilbert 
functions.}}

\author{Alessandro De Stefani}

\date{  }

\begin{document}

\maketitle

\begin{abstract}
In this paper we study Hilbert functions and isomorphism classes of Artinian level local algebras via Macaulay's inverse system. Upper and lower bounds concerning numerical functions admissible for level algebras of fixed type and socle degree are known. For each value in this range we exhibit a level local algebra with that Hilbert function, provided that the socle degree is at most three. Furthermore we prove that level local algebras of socle degree three and maximal Hilbert function are graded. In the graded case the extremal strata have been parametrized by Cho and Iarrobino. 
\end{abstract}

\section{Introduction}

Throughout this paper we consider Artinian local (or graded) $k$-algebras, where $k$ is an algebrically closed field of characteristic zero.

Recall that if $(A,\got{m}, k)$ is an Artinian local $k$-algebra, then there exists an integer $s$ such that $\got{m}^s \ne0$ and $\got{m}^{s+1} = 0$. This integer $s$ is called {\it socle degree} of $A$  and  $A$ is said to be {\it $s$-level of type $\Tt$} if
\begin{equation*}
{\rm Soc}(A) := 0:_A \got{m} = \got{m}^s \ \ \mbox{ and } \ \ \dim_k {\rm Soc}(A) = \Tt.
\end{equation*}
An Artinian local algebra is Gorenstein if it has type $\Tt =1$, and Gorenstein algebras are clearly level. Level algebras can be also defined in higher dimension and, in this case,  $A$ is level if it is Cohen-Macaulay and an Artinian reduction (hence all Artinian reductions) of $A$ is level.  The definition extends also to modules, see for example \cite{Boij} and \cite{Soder}.

The Hilbert function of $A$ in degree $j \gs 0$ is $HF_A(j) = \dim_k  \got{m}^j/  \got{m}^{j+1}$ and  it is also the Hilbert function of the corresponding associated graded ring $G= \oplus_{j\gs 0}  \got{m}^j/  \got{m}^{j+1} $ which is a standard graded $k$-algebra. The sequence $(h_0,h_1,\ldots,h_s)$, where $s$ is the socle degree of $A$ and  $h_j  = HF_A(j)$  for each $j=1,\ldots,s$, is called {\it h-vector} of $A.$  For the aims of this paper, $h$-vector means a numerical sequence which comes from the Hilbert function of some algebras, equivalently  it satisfies the conditions of Macaulay's Theorem.  
\vspace{2mm}

In 1916 Macaulay establishes a one-to-one correspondence, called Macaulay's inverse system, between Artinian Gorenstein algebras and suitable polynomials. This correspondence has been deeply studied in the graded case, among other authors, by A. Iarrobino in a long series of papers, and it can be generalized to Artinian $s$-level local algebras. We present the main facts that will be used in this paper. Some of them are well known, nevertheless they are included in Section 2 for sake of completeness.

The use of Macaulay's inverse system gives an effective approach to the study of isomorphism classes of Artinian algebras. We apply this method to Artinian level local algebras with maximal Hilbert function, which are called {\it compressed level}, that were deeply studied by Iarrobino \cite{Iarr} and Fr{\"o}berg-Laksov \cite{FRLK}. One of the main results in this paper is Theorem \ref{mainteor2}:

\begin{teo}
Let $A$ be a compressed level local algebra of socle degree $3$. Then $A$ is graded, that is $A \simeq G$.
\end{teo}

This theorem extends a result  proved by J. Elias and M.E. Rossi in \cite{RossiElias} in the case of Gorenstein algebras of socle degree $3.$ 
\vspace{2mm}

Macaulay's inverse system gives also a way to read the Hilbert function of an Artinian algebra. Macaulay characterized the Hilbert functions admissible for graded algebras, but very few is known if we restrict the investigation to special classes. Usually the problem becomes even more difficult if we consider classes of local algebras. In fact, in general, the associated graded ring $G$ does not reflect good properties of the original local algebra $A$. For instance $G$ is not necessarily level even if $A$ is level.

A. Iarrobino and independently G. Valla characterize the $h$-vectors admissible for level {\it graded} algebras of codimension two. This result inspired the work of V. Bertella \cite{Bertella} who characterizes the $h$-vectors of level {\it local} algebras of codimension two. In higher codimension, except for a result by Stanley characterizing $h$-vectors of Gorenstein graded algebras of codimension three \cite{Stan}, very little is known for level algebras of type $\Tt \gs 2$. A. Geramita and others  develop in  \cite{GeramitaMigliore} several methods for studying the Hilbert function of graded level algebras, but they give a list of admissible $h$-vectors which completes the investigation only for socle degree $s \ls 5$, or socle degree $s=6$ of type $\Tt=2$. Other results of Boij, Migliore, Mir\`o Roig, Nagel and Zanello enlarge the interest concerning this topic and    new and interesting  methods are introduced (see \cite{Zanello} or \cite{MiglioreZanello}).  Cho and Iarrobino prove lower and upper bounds  for the Hilbert functions of $s$-level standard graded $k$-algebras of given type and they describe the extremal strata (see \cite{ChoI}, Theorem 1.8).

In the local setting the literature is not so rich. For instance the $h$-vectors admissible for Gorenstein local algebras of codimension three are unknown, even if we restrict to complete intersections.

In Theorem \ref{mainteor} we characterize the $h$-vectors of level local $k$-algebras with socle degree $s \ls 3$ and of any codimension. 
\begin{teo}
Let $H=(1,m,n,\Tt)$ be an $h$-vector. Then $H$ is the $h$-vector of an Artinian level local algebra if and only if $n \ls \Tt m$.
\end{teo}

This result  does not hold if we additionally require the algebra to be graded. In this setting indeed a possible characterization is still unknown. The proof we give is constructive and, in some cases, our strategy leads to an Artinian algebra which is actually graded, or even  generated by monomials. This aspect could become relevant for studying problems concerning Lefschetz properties (see \cite{Zanello} and \cite{MiglioreZanello}) or lifting algebras to zero dimensional schemes by deformation arguments.
\vskip 2mm
All examples and computations in this paper are made by using CoCoA \cite{Cocoa}.

\section{Macaulay's inverse system}
In this section we collect  some results on Macaulay's inverse system. In particular we translate the study of the Hilbert functions  and of the isomorphism classes of level local $k$-algebras  in terms of the polynomials of the inverse system. Most of the results are well known in the literature especially in the graded setting, while in the case of local algebras this approach was not explored as much. We refer to  the paper of Emsalem \cite{Emsalem}  and, in the Gorenstein case, to the paper of Iarrobino \cite{Iarrobino} and of Elias and  Rossi \cite{RossiElias}.   

 \par 
 \vskip 3mm
Throughout this paper $R = k \lb x_1,\ldots, x_m \rb$ will be the formal power series ring over an algebraically closed field $k$ of characteristic zero, while $P = k[x_1,\ldots,x_m]$ will be the polynomial ring over $k$ with the same number of variables as $R$. For $i \in N$ we will use the notation $P_{\ls i}:= \{f \in P: \deg f \ls i\}$ and $P_{<i}:=\{f \in P: \deg f <i\}$.

\subsection{Macaulay's correspondence and the Hilbert Function}
The polynomial ring $P$ is an $R$-module under  the following action:
\begin{eqnarray*}
\circ : R \times P \to  \ \ \ \ P \ \ \ \ \ \ \ \ \ \ \ \ \ \\
(g,f) \mapsto g(\partial_1,\ldots,\partial_m)(f)
\end{eqnarray*}
where $\partial_i$ denotes the formal partial derivative with respect to $x_i$.  
Starting from the $R$-module structure of $P$ define the following bilinear map:
\begin{eqnarray} \label{dual}
\begin{array}{c}
\langle \ , \ \rangle \ : R \times P \to \ \ k\ \ \ \ \hskip 1cm \ \ \  \ \ \  \\
(g,f) \mapsto (g \circ f)(0)
\end{array}
\end{eqnarray} 
Let $I$ be a zero dimensional ideal of $R$ (i.e. such that $R/I$ is an Artinian algebra), then
\begin{equation*}
I^\bot := \{f \in P :  \langle I,f\rangle = 0\}
\end{equation*}
is a finitely generated $R$-submodule of $P$. Assume $I^\bot = \langle f_1,\ldots,f_t \rangle_R$, where $\langle\ldots \rangle_R$ denotes the $R$-submodule of $P$ generated by a set of polynomials.  Because of the definition of   $\circ$, $I^\bot$ is stable under taking derivatives and it  is  a  $k$-vector space of dimension the length of $R/I$. Also, as a $k$-vector space, $I^\bot$ is generated by $f_1,\ldots,f_t$ and by all their derivatives with respect to the variables $x_1,\ldots,x_m$. More explicitly, given $\underline{f} = \{f_1,\ldots,f_t\}$ and denoting with $\partial^h \underline{f}$ the set of all the $h$-th derivatives of $f_1,\ldots,f_t$ (we will write $\partial \underline{f}$ for $\partial^1 \underline{f}$), we get: 
\begin{equation*} 
I^\bot = \langle \underline{f}\rangle_R = \langle \partial^h \underline{f}  : k = 0,\ldots, \deg(f) \rangle_k.
\end{equation*}
Here $\langle \ldots\rangle_k$ means the $k$-vector space generated by a set of polynomials. If $M$ is a finitely generated $R$-submodule of $P$, then one can define the following ideal of $R$:
\begin{equation*}
Ann_R(M) := \{g \in R : \langle g,M\rangle = 0\}.
\end{equation*}

\begin{rmk}
Using the properties of $I$ as an ideal of $R$, one can prove that
$
I^\bot = \{f \in P : I \circ f =~0\}.
$
Similarly, since $M$ is a $R$-submodule of $P$, one has $Ann_R(M) = \{g \in R : g \circ M = 0\}$.
\end{rmk}
Macaulay's inverse system establishes a {\it  one-to-one correspondence between zero dimensional ideals $I$ of $R$ and finitely generated $R$-submodules $M$ of $P  $}  sending $I$ to $I^\bot$ and conversely $M$  to $Ann_R(M)$. We can restrict this  correspondence to ideals $I$ of $R$ such that $R/I$ is $s$-level of type $\Tt$. This   theory  is  well known and it is a consequence of Matlis duality.  A crucial fact is the following lemma whose proof is a natural recasting  of Lemma 1.1 in \cite{Iarrobino}. 
\begin{lemma}
\label{lemmainv1}
Let $s,\Tt$ be positive integers and let $\varphi : R \to k^\Tt$ be a $k$-linear homomorphism such that $\varphi(\call{M}^s) = k^\Tt$ and $\varphi(\call{M}^{s+1}) = 0$, where $\call{M} = (x_1,\ldots,x_m)$ denotes the maximal ideal of $R$. Consider $I = \{h \in R : \varphi(Rh) = 0\}$, then $R/I$ is an Artinian $s$-level local algebra of type $\Tt$.
\end{lemma}
 \vskip 3mm
 
Consider now the $R$-module $M = \langle f_1,\ldots,f_\Tt \rangle_R$ with $\deg(f_i)=s$ for all $i$ such that the corresponding leading forms, say $F_1,\ldots,F_\Tt$ (so that $f_i=F_i + \dots $ terms of lower degree),  are linearly independent over $k$.  Consider
\[
\xymatrixcolsep{1mm}
\xymatrixrowsep{1mm}
\xymatrix{
\varphi:  &R \ar[rr] && k^\Tt \\
&r \ar@{|->}[rr] && ((r \circ f_1)(0),\ldots,(r \circ f_\Tt)(0))
}
\]
Then $\varphi(\call{M}^{s+1}) = 0$ and $\varphi(\call{M}^s) = k^\Tt$. This map leads to the following.  
\begin{pro}
\label{invsyscorr2}
There is a one-to-one correspondence between zero dimensional ideals of $R$ such that $R/I$ is $s$-level of type $\Tt$ and $R$-submodules of $P$ generated by $\Tt$ polynomials of degree $s$ having linearly independent forms of degree $s$. The correspondence is defined as follows:
\begin{eqnarray*}
\left\{  \begin{array}{cc} I \subseteq R \mbox{ such that } R/I \\
\mbox{ is Artinian level of type} \\
\Tt \mbox{ and socle degree } s \end{array} \right\}  \ \stackrel{1 - 1}{\longleftrightarrow} \ 
\left\{ \begin{array}{cc} M \subseteq P \mbox{ submodule generated by} \\
\Tt \mbox{ polynomials of degree} \\
s  \mbox{ with l.i. forms  of degree } s \end{array} \right\} \\
I \ \ \ \ \ \ \ \ \longrightarrow \ \ \ \ \ I^\bot  \ \ \ \ \ \ \ \ \qquad \ \ \ \  \ \ \ \qquad \ \ \ \  \ \ \ \\
Ann_R(M) \ \ \ \ \  \longleftarrow \ \ \ \ \  M \ \ \ \ \ \ \ \ \ \ \ \ \ \ \ \  \qquad \qquad \ \ \ \  \ \ \
\end{eqnarray*}
\end{pro}

We omit the proof because this fact is well known (see \cite{Emsalem} and \cite{Iarrobino}). 

\vskip 3mm 
It is also possible to use Macaulay's inverse system in order to get information on Hilbert functions. Let $M$ be a $R$-submodule of $P$, we define  the following finitely generated $k$-vector spaces
\begin{equation*}  
M_i := \frac{M \cap P_{\ls i} + P_{<i}}{P_{<i}} \ \ \ \ \ \mbox{for all } i \gs 0.
\end{equation*}
The correspondence  between zero dimensional ideals $I$ of $R$ and finitely generated $R$-submodules $M$ of $P  $ preserves lengths and Hilbert functions,  that is
\begin{equation}
\label{HF}
HF_{R/I}(i) = \dim_k ((I^\bot)_i) \ \ \ \mbox{for all } i \gs 0.
\end{equation}

\medskip
\subsection{Q-decomposition of the associated graded ring}

Let $A = R/I$ be an Artinian $s$-level local $k$-algebra and let $G=\oplus_i \got{m}^i/\got{m}^{i+1}$ be its associated graded ring. It is well known that the associated graded ring $G$ of $A$ can be presented as the quotient of the polynomial ring $P$ by a homogeneous ideal $I^*$, called initial ideal of $I$. More explicitly $I^*$ is the ideal generated by all the initial forms of elements in $I$. See \cite{RossiValla} for more references, or \cite{Eis} for a more geometrical approach involving the tangent cone.

Notice that, even if we are assuming that $A$ is level, $G$ may no longer be level. In order to study this pathology, starting from the work of A. Iarrobino for Gorenstein algebras (see  \cite{Iarrobino}), let us consider the $\got{m}$-adic filtration $\{ \got{m}^i\}_{i \gs 0} $ and the L\"oewy's filtration $\{ 0:\got{m}^i\}_{i \gs 0}$ on $A$ and define, for each $a \in \{0,\ldots,s+1\}$ and for every $i \gs 0$, the following ideals of $G$   
\begin{equation*}
C(a) = \bigoplus_{i \gs 0} C(a)_i
\end{equation*}
whose homogeneous components can be described as follows:  \\
\begin{equation*}
C(a)_i = \frac{(0:\got{m}^{s+1-a-i}) \cap \got{m}^i}{(0:\got{m}^{s+1-a-i})\cap \got{m}^{i+1}} \subseteq G_i.
\end{equation*}
An immediate consequence of the definition is
\begin{equation*}
G = C(0) \supseteq C(1) \supseteq \ldots \supseteq C(s) = 0.
\end{equation*}
Moreover it is easy to prove that if $a \gs 1$ then $C(a)_i = 0$ for each $i \gs s-a$, while $C(0)_i = 0$ for each $i > s$. 
Define $Q(a):=C(a)/C(a+1)$. Then 
\begin{equation*}
\{Q(a) : a=0,\ldots,s-1\}
\end{equation*}
is called {\it Q-decomposition} of the associated graded ring $G$.
\begin{pro}
\label{q0level}
Let $(A,\got{m},k)$ be an Artinian $s$-level local algebra of type $\Tt$. Then $Q(0) = G/C(1)$ is the unique  Artinian graded $s$-level quotient of $G$ of type $\Tt$ up to isomorphism.
\end{pro}
\begin{proof}[\scshape Proof]
First we prove that $Q(0)$ is $s$-level. Being $\got{m}^{s+1}=0, $ we have $\got{m}^{i+1} \subseteq (0:\got{m}^{s-i})$.
\begin{eqnarray*}
Q(0) = \frac{\bigoplus_{i \gs 0} \got{m}^i / \got{m}^{i+1}}{\bigoplus_{i \gs 0} (0:\got{m}^{s-i})\cap \got{m}^i /(0:\got{m}^{s-i}) \cap \got{m}^{i+1}} = \ \ \ \ \ \ \ \    \\
= \frac{\bigoplus_{i \gs 0} \got{m}^i / \got{m}^{i+1}}{\bigoplus_{i \gs 0} (0:\got{m}^{s-i})\cap \got{m}^i / \got{m}^{i+1}} \simeq \bigoplus_{i \gs 0} \frac{ \got{m}^i}{(0:\got{m}^{s-i})\cap \got{m}^i}.
\end{eqnarray*}
Also $(0:Q(0)_1) = Q(0)_s = G_s$, in fact clearly $Q(0)_s \subseteq (0:Q(0)_1)$. Conversely let $\overline a \in Q(0)_i $ be such that $\overline a \in {\rm Soc}(Q(0)) $ with $i < s.$ Then  $a \in ((0:\got{m}^{s-i-1}):\got{m}) = (0:\got{m}^{s-i})$, that is $\overline{a} = 0$, and so $Q(0)$ is $s$-level. Then $Q(0) $ is the unique  Artinian graded $s$-level quotient of $G$ of type $\Tt$  from  classical facts (see \cite{FRLK}). Indeed, an $s$-level  standard graded algebra only depends on its homogeneous component of degree $s$, and in this case $Q(0)_s = G_s$ does not depend on the decomposition.
\end{proof}

Let us set some notation. Given a polynomial $f \in P=k[x_1,\ldots,x_m]$ of degree $s$, we will denote by $F$ its leading form, that is we can write $f$ as
\begin{equation*}
f= F + \ldots \mbox{ terms of lower degree }  
\end{equation*}
Also, given a set of polynomials $\underline{f}=\{f_1,\ldots,f_\Tt\}$, we will denote by $\underline{F}$ the set $\{F_1,\ldots,F_\Tt\}$ of the corresponding leading forms. According to Section 2 it is defined 
$$
Ann_R(\underline{f}) := Ann_R(\langle f_1,\ldots,f_\Tt\rangle_R) \ \ \mbox{ and } \ \ Ann_R(\underline{F}) := Ann_R(\langle F_1,\ldots,F_\Tt\rangle_R).
$$
We also introduce the following short notation for the quotients:
\begin{equation*}
A_{\underline{f}} := R/Ann_R(\underline{f}) \ \ \ \mbox{ and } \ \ \ A_{\underline{F}} := R/Ann_R(\underline{F}).
\end{equation*}
 \vskip 3mm

\begin{pro}
\label{q0omog}
Let $A_{\underline{f}}$ be an Artinian $s$-level local $k$-algebra of type $\Tt$. Then
\begin{equation*}
Q(0) \simeq A_{\underline{F}}.
\end{equation*}
\end{pro}
\begin{proof}[\scshape Proof]
Notice that, using the inclusion $P \subseteq R$ and restricting the homomorphism $\circ$ defined in Section 2, one can view $P$ as a module over itself (this structure of $P$-module is different from the natural one given by the ring structure). Also, given the set of homogeneous polynomials $\underline{F} = \{F_1,\ldots,F_\Tt\}$ and using this restriction, one can define the set $Ann_P(\langle \underline{F} \rangle_P) \subseteq P$. It turns out that $Ann_P(\langle \underline{F} \rangle_P)$ is a homogenous ideal of $P$, since $F_1,\ldots,F_\Tt$ are all homogeneous, and furthermore the quotient $P/ Ann_P(\langle \underline{F} \rangle_P)$ is an Artinian graded $s$-level $k$-algebra of type $\Tt$. Also, using once again the fact that $F_1,\ldots,F_\Tt$ are homogeneous, one has
\[
A_{\underline{F}} = \frac{R}{Ann_R(\langle \underline{F} \rangle_R)} \simeq \frac{P}{Ann_P(\langle \underline{F} \rangle_P)}.
\]
\begin{claim*}
Let $I := Ann_R(\langle \underline{f}\rangle_R)$. Then $I^* \subseteq Ann_P(\langle \underline{F}\rangle_P)$.
\end{claim*} 
The proposition follows from the claim, in fact assuming $I^* \subseteq Ann_P(\langle \underline{F}\rangle_P)$ we have
\begin{equation*}
\frac{P}{Ann_P(\langle \underline{F}\rangle_P)}  \simeq  \frac{P/I^*}{Ann_P(\langle \underline{F}\rangle_P)/I^*}  \simeq \frac{G}{Ann_P(\langle \underline{F}\rangle_P)/I^*}.
\end{equation*}
But $Q(0)$ is the only graded $s$-level quotient of $G$ of type $\tau$ up to isomorphism by Proposition \ref{q0level}, and $P/Ann_P(\langle \underline{F}\rangle_P)$ is also graded $s$-level of type $\Tt$. Therefore we must have
\begin{equation*}
A_{\underline{F}} \simeq P/Ann_P(\langle \underline{F}\rangle_P) \simeq Q(0).
\end{equation*}
\begin{proof}[\scshape Proof of the Claim]
We prove first that $I^*_s \subseteq \left[Ann_P(\langle \underline{F}\rangle_P)\right]_s$, where the subscript means the homogeneous component of the two ideals. Let $H \in I^*_s$, then there exists $h \in R$ such that $h^* = H$, that is $h = H + \dots$ higher order terms, and also $h \in I = Ann_R(\langle \underline{f}\rangle_R)$. For $i \in \{1,\ldots,\Tt\}$ we get
\[
\xymatrixrowsep{0.5mm}
\xymatrixcolsep{5mm}
\xymatrix{
0 = (h \circ f_i)(0) = \left[(H + \mbox{ terms of order} > s) \circ (F_i + \mbox{ terms of degree} < s)\right](0) = (H \circ F_i)(0),  
}
\]
and this equality holds for each $i = 1,\ldots,\tau$. Therefore 
\begin{equation*}
H \in \left[Ann_P(\langle \underline{F}\rangle_P)\right]_s
\end{equation*}
and  we get $I^*_s \subseteq \left[Ann_P(\langle \underline{F}\rangle_P)\right]_s$. By a well known duality result for graded level algebras (see for example \cite{FRLK}), for all $t \in \g{N}$ we have
\[
\left[Ann_P(\langle \underline{F}\rangle_P)\right]_t  = \left\{ D \in P_t : D P_{s-t} \subseteq \left[Ann_P(\langle\underline{F}\rangle_P)\right]_s\right\}.
\end{equation*}
Let $D \in I^*_t$ and let $E \in P_{s-t}$, then $DE \in I^*_s \subseteq \left[Ann_P(\langle\underline{F}\rangle_P)\right]_s$. This holds for all $E \in P_{s-t}$, therefore $D \in \left[Ann_P(\langle \underline{F}\rangle_P)\right]_t$. Since $D\in I^*_t$ and $t \in \g{N}$ are arbitrary, this proves the containment and hence the claim.
\renewcommand{\qedsymbol}{}
\end{proof}
\end{proof}

\begin{pro}
\label{equivfactlev} 
Let $A_{\underline{f}}$ be an Artinian $s$-level local $k$-algebra of type $\Tt$. The following facts are equivalent:
\begin{compactenum}[\rm (i)]
\item $G$ is $s$-level of type $\Tt$.
\item $G \simeq Q(0)$.
\item $C(1)=0$.
\item $C(a) = 0$ for all $a \gs 1$.
\item $Q(a) = 0$ for all $a \gs 1$.
\end{compactenum}
\end{pro}
\begin{proof}[\scshape Proof]
By Proposition \ref{q0level} $Q(0)=G/C(1)$ is the only graded $s$-level quotient of $G$ of type $\Tt$. Then (i) implies (ii), which is clearly equivalent to (iii). Moreover, being $C(1) \supseteq C(2) \supseteq \ldots C(s) = 0$, (iii) implies (iv). Because of the definition of the $Q(a)$'s, if $C(a) = 0$ for all $a \gs 1$, then (v) holds. Finally, (v) implies (i) by Proposition \ref{q0level}.
\end{proof}
Let us focus now on level algebras of socle degree $s=3$. We have $C(a) = 0$ for all $a \gs 2$, therefore:
\begin{equation*}
G = C(0) \supseteq C(1) \supseteq C(2) = 0.
\end{equation*}
Moreover
\begin{equation}
\label{glevel}
C(1)_1 = \frac{0: \got{m}^2}{(0: \got{m}^2)\cap \got{m}^2} \ \mbox{ while } \ C(1)_i = 0 \ \mbox{ for each } i\ne 1.
\end{equation}
\begin{coroll}
Let $A_{\underline{f}}$ be an Artinian 3-level local algebra.. Then its associated graded algebra $G$ is 3-level of type $\Tt$ if and only if $\got{m}^2 = 0: \got{m}^2$.
\end{coroll}
\begin{coroll}
\label{corc1}
Let $A_{\underline{f}}$ be an Artinian 3-level local $k$-algebra. Then, with the above notation
\begin{eqnarray*}
HF_{A_{\underline{f}}}(j) = \left\{
\begin{array}{ll}
HF_{Q(0)}(j) + HF_{C(1)}(j) \ \ \mbox{ if } j =1 \\
HF_{Q(0)}(j) \ \ \ \ \ \ \qquad \qquad \mbox{ if } j=0,2,3
\end{array} \right.
\end{eqnarray*}
\end{coroll}
\medskip

Next proposition shows that the Q-decomposition is also very useful to establish whether a local algebra is graded or not.

\begin{pro}
\label{procangrad}
Let $A_{\underline{f}}$ be an Artinian $s$-level local algebra of type $\Tt$. The following facts are equivalent:
\begin{compactenum}[\rm (i)]
\item $A_{\underline{f}}$ is graded.
\item $A_{\underline{f}} \simeq Q(0)$.
\item $\langle \underline{f}\rangle_R  \simeq  \langle \underline{F}\rangle_R$ as $R$-modules.
\end{compactenum}
\end{pro}
\begin{proof}[\scshape Proof]
(i) $\Rightarrow$ (ii) If $A_{\underline{f}}$ is  graded then $G$ is level of type $\tau$. Up to isomorphism $Q(0)$ is the only level quotient of $G$ by Proposition \ref{q0level}, then $A_{\underline{f}} \simeq G \simeq Q(0)$. \\
(ii) $\Rightarrow$ (iii) It follows immediately by Proposition \ref{q0omog}. \\
(iii) $\Rightarrow$ (i) Being $\langle \underline{f}\rangle_R \simeq \langle \underline{F}\rangle_R$ one has $A_{\underline{f}} \simeq Q(0) = G/C(1)$, therefore $HF_{A_{\underline{f}}} = HF_{G/C(1)}$.\\
Moreover there exists an epimorphism $G \stackrel{\pi}{\longrightarrow} G/C(1)$ given by the natural projection on the quotient. By definition $HF_{A_{\underline{f}}} =HF_G$ and so $HF_G = HF_{G/C(1)}$, therefore $C(1) = 0$ and $A_{\underline{f}} \simeq G$ is  graded.
\end{proof}

\subsection{Isomorphism classes of Artinian level algebras}

In this section we present a possible approach for studying the classical problem of the isomorphism classes of local algebras using the tool of Macaulay's inverse system. We essentially follow the approach used by J. Elias and M.E. Rossi in \cite{RossiElias} for Gorenstein algebras and we extend it to level algebras.
\vskip 3mm

Given an Artinian  local $k$-algebra $A$ denote by $Aut_a (A)$ and $Aut_v(A)$ the group of the automorphisms of $A$ as a $k$-algebra and as a $k$-vector space respectively. The automorphisms of the power series ring $(R, \call{M},k)$ as  a $k$-algebra are well known (see for example \cite{Singular}, Theorem 6.2.18). If $\call{M}= (x_1,\ldots, x_m) $ they act as a replacement $x_i \mapsto z_i $, $i=1,\cdots, n$, such that $\call{M}= (x_1,\ldots, x_m)= (z_1, \dots, z_m)$. We are interested in Artinian local algebras $A=R/I$ with socle degree $s$ and, since $\call{M}^{s+1}  \subseteq I$, we can restrict to the $k$-algebra automorphisms of $R/ \call{M}^{s+1} $ induced by the projection $\pi:R\longrightarrow R/\call{M}^{s+1}.$ Clearly  $Aut_a(R/ \call{M}^{s+1})  \subseteq Aut_v (R/ \call{M}^{s+1})$.

Let $E=\{ e_i \} $ be the canonical basis of $R/\call{M}^{s+1}$ as a $k$-vector space consisting of the standard monomials $x^{\alpha}$ ordered by the deg-lex order with $x_1>\cdots >x_m$.
Then the dual basis of $E$ with respect to the perfect paring $\langle \ , \ \rangle$ defined in (\ref{dual}) is the basis $E^*=\{ e_i^* \} $ of $P_{\ls s}$,  where
\begin{equation*}
(x^{\alpha})^* = \frac 1 {\alpha !} y^{\alpha},
\end{equation*}
in fact $e_i^* (e_j)= \langle e_j ,  e_i^* \rangle=\delta_{ij}$, where $\delta_{ij}=0$ if $i \ne j$ and $\delta_{ii}=1$. Hence for any $\varphi  \in  Aut_v(R/\call{M }^{s+1}) $ we can associate a matrix $M(\varphi)$ with respect to the basis $E$ of size $l = \dim_k (R/\call{M}^{s+1}) =  \binom{m+s}{s}$. We have  the following natural sequence of morphisms of groups: 
\begin{equation} 
\label{groups} 
Aut_a(R) \stackrel{\pi}{\longrightarrow} Aut_a(R/\call{M}^{s+1} ) \stackrel{\sigma}{\longrightarrow} Aut_v(R/\call{M}^{s+1} ) \stackrel{\rho_E}{\longrightarrow} Gl_l (k). 
\end{equation}

\noindent Given $I$ and $J$ ideals of $R$ such that $\call{M}^{s+1}\subseteq I, J$, there exists a $k$-algebra isomorphism
\begin{equation*}
\varphi : R/I \to R/J
\end{equation*}
if and only if $ \varphi$ is canonically induced by a $k$-algebra automorphism of $R/\call{M}^{s+1}$ sending $I/\call{M}^{s+1}$ to $J/\call{M}^{s+1}$. In particular $\varphi$ is an isomorphism of $k$-vector spaces. Dualizing
\begin{equation*}
\varphi^* : (R/J)^* \to (R/I)^*
\end{equation*}
is an isomorphism of the $k$-vector subspaces $(R/I)^* \simeq I^{\bot}$ and $(R/J)^* \simeq J^{\bot}$ of $P_{\ls s}$. Hence $^t M (\varphi)$ is the matrix associated to $\varphi^*$ with respect to the basis $E^*$ of $P_{\ls s}$. The following commutative diagram helps to visualize our setting:
\begin{eqnarray*}
\begin{array}{ccc}
Aut_{v}(R/\mathcal M^{s+1} ) &  \stackrel{\rho_E}{\longrightarrow} & Gl_l (k) \\
\downarrow *&  & \; \downarrow \; ^t() \\
Aut_{v}(P_{\ls s} )          &  \stackrel{\rho_{E^*}}{\longrightarrow} & Gl_l (k) \\
\end{array}
\end{eqnarray*}

\noindent
Denote by $\call{R}$ the  subgroup of  $Aut_v(P_{\ls s})$ represented by the matrices $^t M (\varphi)$ of  $Gl_l (k)$ with $\varphi\in Aut_a(R/\call{M}^{s+1})$. The following theorem by Emsalem holds.

 \begin{teor}
 \label{Emsalem}{\rm{ (\cite[Proposition 15]{Emsalem})}}
The classification, up to isomorphism, of the Artinian local $k$-algebras of multiplicity $e$, socle degree $s$ and embedding dimension $m$ is equivalent to the classification, up to the action of $\call{R}$, of the $k$-vector subspaces of $P_{\ls s} = \{g \in P: \deg g \ls s\}$ of dimension $e$, stable under derivation and containing $P_{\ls 1}$.
\end{teor}

In the special case of Gorenstein algebras, given $\varphi \in Aut_a(R)$, one has
\begin{equation*}
\varphi(A_f) = A_g \ \mbox{ if and only if } \ (\varphi^*)^{-1}(\langle f \rangle_R) = \langle g\rangle_R.
\end{equation*}
In \cite{RossiElias} J. Elias and M.E. Rossi proved that, since a generator of a cyclic $R$-module is defined modulo an invertible element of $R$ ($f$ and $g$ must be polynomials of the same degree), if $\varphi(A_f) = A_g$ then $(\varphi^*)^{-1}(f) = u \circ g$, where $u$ is a unit in $R$. With level algebras we have
\begin{equation*}
\varphi(A_{\underline{f}}) = A_{\underline{g}} \ \mbox{ if and only if } \ (\varphi^*)^{-1} (\langle f_1,\ldots,f_\Tt \rangle_R) = \langle g_1,\ldots,g_\Tt \rangle_R.
\end{equation*}
If $\varphi(A_{\underline{f}}) = A_{\underline{g}}$, in analogy with the Gorenstein case, we get the following result.
\begin{lemma}
\label{lemmaconv}
Let $\varphi \in Aut_a(R)$. The following facts are equivalent: 
\begin{compactenum}[\rm (i)]
\item $\varphi(A_{\underline{f}}) = A_{\underline{g}}$.
\item There exists $B \in Gl_\Tt(R)$ such that $^t((\varphi^*)^{-1}(f_1),\ldots,(\phi^*)^{-1}(f_\Tt)) = B \circ \ ^t(g_1,\ldots,g_\Tt)$.
\end{compactenum}
\end{lemma}

\begin{rmk} In the graded case $(\varphi^*)^{-1}$ is a linear transformation which acts by substitution on the variables of the homogeneous polynomials $F_1,\ldots,F_\Tt$. In the local case this fact is true if $\varphi$ is a linear transformation, but in general it is no longer true.
\end{rmk}

Let $B \in Gl_\Tt(R/\call{M}^{s+1})$, then this matrix acts on $ (P_{\ls s})^\Tt$ as  an isomorphism of $k$-vector spaces. Set $l = \dim_k P_{\ls s}$, then there exists a matrix $N(B) \in Gl_{l \Tt}(k)$ associated to the action of $B$ on $(P_{\ls s})^\Tt$ with respect  to the basis $(E^*)^{\oplus \Tt}$ of $(P_{\ls s})^\Tt$. 

Let $f_1,\ldots,f_\Tt \in P_{\ls s}$, then set $f_i = b_{i1}e_1^* + \ldots + b_{il}e_l^* \in P_{\ls s}$ for all $i=1,\ldots,\Tt$. We denote the row vector of the coefficients of the $\Tt$ polynomials with respect to the basis $E^*$ by\\
\begin{equation*}
[\underline{f}]_{E^*} = (b_{11},\ldots,b_{1l}, \dots, b_{\Tt1},\ldots,b_{\Tt l}).
\end{equation*}
This leads to a natural generalization of Proposition $2.2$ in \cite{RossiElias}, which is stated for Gorenstein algebras instead.
\begin{pro}
\label{classisolevel}
Two Artinian $s$-level algebras $A_{\underline{f}}$ and $A_{\underline{g}}$ of type $\Tt$ are isomorphic if and only if there exists $\varphi \in Aut_a(R/\call{M}^{s+1})$ and there is an invertible matrix $B \in Gl_{\Tt}(R/\call{M}^{s+1})$ such that 
\begin{equation*}
[\underline{g}]_{E^*} (^t N(B) M(\varphi^{\oplus \Tt})) = [\underline{f}]_{E^*},
\end{equation*}
where $M(\varphi^{\oplus \Tt})$ is the matrix associated to the homomorphism which consists of $\Tt$ copies of $\varphi$.
\end{pro}

The above result allows us to translate the difficult problem of the classification of level local algebras up to isomorphism into a problem of linear algebra. 
 
\section{Compressed level local algebras which are graded}

If $A_f$ is a Gorenstein algebra with $h$-vector $H=(1,m,n,1)$, J. Elias and M.E. Rossi proved in \cite{RossiElias}  that $A_f$ is graded if and only if $H$ is admissible as the $h$-vector of a graded Gorenstein algebra. This no longer true if $A$ is level. For instance $H=(1,4,5,6)$ is achieved by a level graded algebra (see Example \ref{es1}), however if we consider 
\begin{equation*}
f_1:= x^3+t^2, \ \ \ f_2:=x^2y, \ \ \ f_3:=xy^2, \ \ \ f_4:=z^3, \ \ \ f_5:=xz^2, \ \ \ f_6:=y^3 
\end{equation*}
$A_{\underline{f}} = R/Ann(\langle f_1,\ldots,f_6\rangle_R)$ has $h$-vector $H$ but it is not isomorphic to its associate graded algebra $G$, because $Q(0)= A_{\underline{F}}$ has $h$-vector $(1,3,5,6)$.  

Another direction of investigation arises from the fact that in the Gorenstein case, as already pointed out, $A$ is graded if and only if its $h$-vector is admissible for a graded Gorenstein algebra, but this fact is also equivalent to say that $A$ has maximal Hilbert function. In fact, if $A$ has $h$-vector $(1,m,n,1)$, then necessarily $n \ls m$, and $A$ is graded if and only if equality holds. More generally, if $A$ is $3$-level of type $\Tt$ with $h$-vector $(1,m,n,\Tt)$, then $n \ls \Tt m$ (see Theorem \ref{mainteor}), and when equality holds $A$ is called {\it compressed level} or, for brevity, just {\it compressed}. We will prove that if $A_{\underline{f}}$ is a compressed level local algebra of socle degree 3, then it is graded. We need some further definitions and results before proving it. Let $F \in P_3$ be a form of degree three. Define:
\begin{compactitem}
\item $\Delta_1(F)$ the $m \times \binom{m+1}{2}$ matrix which $j$-th row is the vector of the coefficients of $\partial_j F$ ordered following lex order on the dual basis $\{\frac{x_1^2}{2},x_1x_2,\ldots,\frac{x_m^2}{2}\}$ of $P_2$.
\item $\Delta_2(F)$ the $\binom{m+1}{2} \times m$ matrix which $j$-th row is the vector of the coefficients of $\partial_{\underline{s}} F$ ordered following lex order on $P_1 = \{x_1,\ldots,x_m\}$. Here $\underline{s} \in \g{N}^m$, with $|s|=2$, is such that $x^{\underline{s}}$ is the $j$-th element of $\g{T}_m^2$ in lex order.
\end{compactitem}
Essentially, $\Delta_1(F)$ and $\Delta_2(F)$ are the matrices of the coefficients (with respect to the dual bases) of the  first and the second derivatives of $F$.

\begin{lemma}
\label{lemmatr}
Let $F \in P_3$ be a form of degree three. Then
\begin{equation*}
\Delta_1(F) =\  ^t \Delta_2(F).
\end{equation*}
\end{lemma}
\begin{proof}[\scshape Proof]
Let $r=1,\ldots,m$. Set $\delta_r \in \g{N}^m$ the vector which entries are all $0$ except for position $r$, in which there is $1$. Also set $F= \sum_{|\underline{i}| = 3}\alpha_{\underline{i}} \ \frac{x^{\underline{i}}}{\underline{i}!}$, written in the dual basis. Notice that
\begin{equation*}
\partial_r F = \sum_{|\underline{j}| = 2}\beta_{\underline{j}} \ \frac{x^{\underline{j}}}{\underline{j}!} = \sum_{|\underline{j}| = 2}\alpha_{\underline{j} + \delta_r} \ \frac{x^{\underline{j}}}{\underline{j}!}.
\end{equation*}
Also
\begin{equation*}
\partial_{\underline{s}} F = \sum_{k=1}^m \gamma_{\delta_k} x_k = \sum_{k=1}^m \alpha_{\delta_k+\underline{s}} \ x_k.
\end{equation*}
Hence the coefficient of $\frac{x^{\underline{s}}}{\underline{s}!}$ in $\partial_r F$ is exactly the coefficient of $x_r$ in $\partial_{\underline{s}} F$, and it is equal to $\alpha_{\delta_r + \underline{s}}$. This means
\begin{equation*}
\Delta_1(F) =\  ^t \Delta_2(F).
\end{equation*}
\end{proof}
By the above lemma, $\Delta_1(F)$ and $\Delta_2(F)$ carry the same information about $F$, therefore we consider only $\Delta(F) := \Delta_1(F)$. Let $A_{\underline{F}}$ be a graded level algebra with $h$-vector $(1,m,n,\Tt)$. We can define:
\begin{eqnarray*}
\Delta(\underline{F}) = \left(
\begin{tabular}{c}
$\Delta(F_1)$ \\
 
\hline
$\vdots$ \\
\hline
$\Delta(F_\Tt)$
\end{tabular} \right)
\end{eqnarray*}
which is a $\Tt m \times \binom{m+1}{2}$ matrix. It is easy to see that
\begin{equation}
\label{maxrk}
rk(\Delta(\underline{F})) = n,
\end{equation}
hence $\Delta(\underline{F})$ has maximal rank if and only if 
$$
n = \min \left\{ \Tt m, \binom{m+1}{2}\right\},
$$
if and only if $A_{\underline{F}}$ is compressed.
\begin{lemma}
\label{q0compr}
Let $A_{\underline{f}}$ be a compressed level local algebra of socle degree three. Then $Q(0) = A_{\underline{F}}$ is compressed.
\end{lemma}
\begin{proof}[\scshape Proof]
Let $(1,m,n,\Tt)$ be the $h$-vector of $A_{\underline{f}}$. Being $A_{\underline{f}}$ compressed we have
\begin{equation*}
n = \min \left\{ \Tt m, \binom{m+1}{2}\right\}.
\end{equation*}
Corollary \ref{corc1} shows that $Q(0)$ has $h$-vector $(1,r,n,\Tt)$, with $r \ls m$. But $Q(0)$ is a level algebra itself and hence its $h$-vector must satisfy the following inequality:
\begin{equation*}
\min \left\{ \Tt m, \binom{m+1}{2}\right\} = n \ls \min \left\{ \Tt r, \binom{r+1}{2}\right\}.
\end{equation*}
This yields $m \ls r$ and hence $r=m$. Therefore $Q(0)$ has $h$-vector $(1,m,n,\Tt)$ and it is a compressed algebra.
\end{proof}
\begin{teor}
\label{mainteor2}
Let $A_{\underline{f}}$ be a compressed level local algebra of type $\Tt$ and socle degree 3. Then $A_{\underline{f}}$ is graded.
\end{teor}
\begin{proof}[\scshape Proof]
First notice that, being $P_{\ls 1} = \{g \in P : \deg g \ls 1\} \subseteq \langle f_1,\ldots,f_\Tt \rangle_R$, we can  assume that $f_i = F_i + Q_i$, for each $i=1,\ldots,\Tt$, where $F_i$ and $Q_i$ are forms of degree $3$ and $2$ respectively. 
Let $(1,m,n,\Tt)$ be the $h$-vector of $A_{\underline{f}}$, then it has to be
\begin{equation*}
n = \min \left\{ \Tt m, \binom{m+1}{2} \right\}
\end{equation*}
since $A_{\underline{f}}$ is compressed. Assume $n=\binom{m+1}{2}$. Then $\dim_k (R_1 \circ \langle \underline{f}\rangle_R) = \binom{m+1}{2}$ and therefore $P_2 \subseteq \langle f_1,\ldots,f_\Tt \rangle_R$ too. In this case:
\begin{equation*}
\langle f_1,\ldots,f_\Tt \rangle_R \ = \ \langle F_1 + Q_1,\ldots,F_\Tt + Q_\Tt\rangle_R \ = \ \langle F_1,\ldots,F_\Tt\rangle_R
\end{equation*}
and hence $A_{\underline{f}}$ is graded by Proposition \ref{procangrad} (iii).

\medskip
Assume now $n=\Tt m$. We want to prove that there exists $\varphi \in Aut_{a}(R/\call{M}^4)$ such that
\begin{equation*}
[\underline{F}]_{E^*} M(\varphi) = [\underline{f}]_{E^*} = [\underline{F} + \underline{Q}]_{E^*}.
\end{equation*}
Consider $\varphi: R/\call{M}^4 \to R/\call{M}^4$ the $k$-algebra automorphism with the identity as Jacobian defined as follows on the variables:
\begin{equation*}
\varphi(x_h) := x_h+\sum_{|\underline{i}| = 2} a^h_{\underline{i}} x^{\underline{i}} \ \ \ \ \ \ h=1,\ldots,m
\end{equation*}
where $a^h_{\underline{i}} \in k$ for each $|\underline{i}| = 2$, $h=1,\ldots,m$ and $\underline{a} := (a^h_{\underline{i}} \ : \ |\underline{i}| = 2, h=1,\ldots,m)$ is a row vector of size $m\binom{m+1}{2}$. The matrix $M(\varphi)$ associated to $\varphi$ is an element of $Gl_r(k)$, where $r = \binom{m+3}{4}$, with respect to the basis $E$ of $R/\call{M}^4$ ordered by the deg-lexicographic order. Explicitly:
\begin{eqnarray*}
M(\varphi) = \left(
\begin{tabular}{c|c|c|c}
1 & 0 & 0 & 0       \\
\hline
0 & $I_m$ & 0 & 0           \\
\hline
 0 & $D$ & $I_{\binom{m+1}{2}}$ & 0    \\
\hline      
0 & 0 & $B$ &          $I_{\binom{m+2}{3}}$
\end{tabular}
\right)
\end{eqnarray*}
where for all $t \gs1$ $I_t$ denotes the $t \times t$ identity matrix. The first block of columns corresponds to the image $\varphi(1) = 1$. The second block of columns corresponds to the image of $\varphi(x_i)$, $i = 1,\ldots,m$. The third block of columns corresponds to the image of $\varphi(x^{\underline{i}})$, where $|\underline{i}| = 2$. Finally the fourth block of columns corresponds to the image of $\varphi(x^{\underline{i}})$ with $|\underline{i}| = 3$, that is the identity matrix. Hence $D$ is the $\binom{m+1}{2} \times m$ matrix defined by the coefficients of the degree two monomials of $\varphi(x_i)$, $i = 1,\ldots,m$ and $B$ is a $\binom{m+2}{3} \times \binom{m+1}{2}$ matrix defined by the coefficients of the degree three monomials appearing in $\varphi(x^{\underline{i}})$, with $|\underline{i}| = 2$. It is clear that
$M(\varphi)$ is determined by $D$ and the entries of $B$ are linear forms in the variables $a^h_{\underline{i}}$, with $|\underline{i}| = 2$, $h = 1,\ldots,m$. Now write $F_1,\ldots,F_\Tt$ in the dual basis $E^*$:
\begin{equation*}
F_j = \sum_{|\underline{i}| = 3}\alpha^j_{\underline{i}} \frac{x^{\underline{i}}}{\underline{i} !} \ \ \ \ \ \ j=1,\ldots,\Tt
\end{equation*}
and also $Q_1,\ldots,Q_\Tt$:
\begin{equation*}
Q_j = \sum_{|\underline{i}| = 2}\beta^j_{\underline{i}} \frac{x^{\underline{i}}}{\underline{i} !} \ \ \ \ \ \ j=1,\ldots,\Tt.
\end{equation*}
Set $\alpha^j := (\alpha_{\underline{i}}^j \ : \ |\underline{i}| = 3 )$ a row vector of size $\binom{m+2}{3}$ and similarly $\beta^j := (\beta_{\underline{i}}^j \ : \ |\underline{i}| = 2 )$ a row vector of size $\binom{m+1}{2}$. We have to solve the following linear system:
\begin{equation*}
[\alpha^j]_{E^*} B = [\beta^j]_{E^*} \ \ \ \ \ j=1,\ldots,\Tt
\end{equation*}
or equivalently:
\begin{equation*}
[\underline{\alpha}]_{E^*} \ B^{\oplus \Tt} = [\underline{\beta}]_{E^*}
\end{equation*}
where $\underline{\alpha} = (\alpha^j \ : \ j=1,\ldots,\Tt)$ is a row vector of size $\Tt \binom{m+2}{3}$ and similarly $\underline{\beta} = (\beta^j \ : \ j=1,\ldots,\Tt)$ is a vector of size $\Tt \binom{m+1}{2}$ and
\begin{eqnarray*}
B^{\oplus \Tt} = \left(
\begin{tabular}{c}
$B$ \\ 
\hline
$B$ \\ 
\hline
\vdots \\
\hline
$B$
\end{tabular}
\right)
\end{eqnarray*}
is a $\Tt \binom{m+2}{3} \times \binom{m+1}{2}$ matrix.
\medskip
J. Elias and M.E. Rossi prove in \cite{RossiElias} that for each $j=1,\ldots,\Tt$ there exists a $\binom{m+1}{2} \times m\binom{m+1}{2}$ matrix $M_j$ such that:
\begin{equation*}
[\alpha^j]_{E^*} B = \underline{a} \ ^t M_j
\end{equation*} where $\underline{a}$ are the coefficients defining the automorphism $\varphi$. Hence we have to solve
\begin{equation*}
\underline{a} \ ^t M = [\beta]_{E^*}
\end{equation*}
where 
\begin{eqnarray*}
M = \left(
\begin{tabular}{c}
$M_1$ \\
\hline
 $M_2$ \\
 \hline
 \vdots \\
 \hline
 $M_\Tt$
\end{tabular}
\right)
\end{eqnarray*}
is a $\Tt \binom{m+1}{2} \times m \binom{m+1}{2}$ matrix.

\medskip
For each $j = 1,\ldots,\Tt$ recall the matrices $\Delta(F_j)$ previously introduced. Elias and Rossi prove in \cite{RossiElias} that the matrices $M_j$ have the following upper-diagonal structure:
\begin{eqnarray*}
M_j = \left(
\begin{tabular}{c|c|c|c|c}
$M^1_j$ & * & \ldots & * & *      \\
\hline
0 & $M^2_j$ & \ldots & * & *           \\
\hline
 $\vdots$ & $\vdots$ & $\vdots$ & $\vdots$ & $\vdots$   \\

\hline      
0 & 0 & \ldots & $M^{m-1}_j$ & * \\
\hline
0 & 0 & \ldots & 0 & $M^m_j$ 
\end{tabular}
\right)
\end{eqnarray*}
where $M^l_j$ is a $(m-l+1) \times \binom{m+1}{2}$ matrix, $l=1,\ldots,m$, such that:
\begin{compactenum}[(i)]
\item $1$-st row of $M^1_j$ equals two times the $1$-st column of $\Delta(F_j)$. \\
\ \ \ \ \ \ $t$-th row of $M^1_j$ equals the $t$-th column of $\Delta(F_j)$, $t=2,\ldots,m$.
\item $1$-st row of $M^l_j$ equals two times the $l$-th column of $\Delta(F_j)$, for $l=2,\ldots,m$. \\
\ \ \ \ \ \ $t$-th row of $M^l_j$ equals the $(l+t-1)$-th column of $\Delta(F_j)$, $t=2,\ldots,m-l+1$, $l=2,\ldots,m$.
\end{compactenum}
In this way they show that $M_j$ has maximal rank because $\Delta(F_j)$ has maximal rank, hence the system $\underline{a} \ ^t M_j = [\beta^j]_{E^*}$ is compatible and this completes their proof. In our generalization to level algebras we have to prove that there is a solution $\underline{a}$ which satisfies $\underline{a} \ ^t M_j = [\beta^j]_{E^*}$ for all $j=1,\ldots,\Tt$, hence we have to show that the matrix $M$ has maximal rank. But in this case: \\
\begin{eqnarray*}
M = \left(
\begin{tabular}{c|c|c|c|c}
$M^1_1$ & * & \ldots & * & *      \\
\hline
0 & $M^2_1$ & \ldots & * & *           \\
\hline
 $\vdots$ & $\vdots$ & $\vdots$ & $\vdots$ & $\vdots$   \\

\hline      
0 & 0 & \ldots & $M^{m-1}_1$ & * \\
\hline
0 & 0 & \ldots & 0 & $M^m_1$  \\
\hline
\hline 
 
 $\vdots$ & $\vdots$ & $\vdots$ & $\vdots$ & $\vdots$   \\
  $\vdots$ & $\vdots$ & $\vdots$ & $\vdots$ & $\vdots$   \\
   $\vdots$ & $\vdots$ & $\vdots$ & $\vdots$ & $\vdots$   \\
 \hline
\hline 
$M^1_\Tt$ & * & \ldots & * & *      \\
\hline
0 & $M^2_\Tt$ & \ldots & * & *           \\
\hline
 $\vdots$ & $\vdots$ & $\vdots$ & $\vdots$ & $\vdots$   \\

\hline      
0 & 0 & \ldots & $M^{m-1}_\Tt$ & * \\
\hline
0 & 0 & \ldots & 0 & $M^m_\Tt$  \\
\end{tabular}
\right).
\end{eqnarray*}
Therefore the rank of $M$ is linked to the rank of $\Delta({\underline{F}})$. But we assumed that $A_{\underline{f}}$ is compressed, and hence $Q(0) = A_{\underline{F}}$ is compressed by Lemma \ref{q0compr}. So the matrix $\Delta({\underline{F}})$ has maximal rank equal to $\Tt m$ by (\ref{maxrk}). This means that the system is compatible and completes the proof.
\end{proof}
The converse of Theorem \ref{mainteor2} does not hold:
\begin{ese}
\label{es1222}
Consider the $h$-vector $(1,2,2,2)$ which is achieved, for example, by the following level local $k$-algebra:
\[
A_{\underline{f}} =  A_{\{x^3,y^3\}} = \frac{k\lb x,y \rb}{(x^4,xy,y^4)} \mbox{.}
\]
$A_{\underline{f}} \simeq A_{\underline{F}}$ is graded, however it is not compressed.
\end{ese}
 
\section{$h$-vectors of Artinian level local algebras}
For the aims of this section $A$ will denote an  Artinian $k$-algebra of socle degree $s$, type $\Tt$ and embedding dimension $m. $ We can assume that $A \simeq k \lb x_1,\ldots,x_m \rb/I$ with $I \subseteq (x_1,\ldots.x_m)^2$.

\subsection{Socle degree two}
 
Note that the case $s = 1$ ($h$-vector $(1,m)$) is trivial, since $I = (x_1,\ldots,x_m)^2$ and $A$ is  level.

\medskip 
\noindent If $s=2, $ then $A$ is not necessarily level. If $A$ is level, then  we can easily prove that $A$ actually is a standard {\it{graded}} $ k$-algebra. This is true because $G$ itself is level since $G=Q(0)$. Recalling the notation introduced in Section 2.2, if we write $A=A_{\underline{f}}$, then  $\langle \underline{f}\rangle_R = \langle \underline{F}\rangle_R $ because $P_{\ls1 } = \{g \in P: \deg g \ls 1\} \subseteq \langle \underline{F}\rangle_R. $

\medskip
\noindent If the socle degree is $s = 2$ all $h$-vectors are admissible for an Artinian level local algebra. This means that any numerical sequence $(1,m,\Tt)$ which satisfies the conditions of Macaulay's Theorem (in this case just $\Tt \ls \binom{m+1}{2}$) can be the $h$-vector of a 2-level local algebra. The result could be already known, nevertheless we present an easy and constructive proof, based on the use of Macaulay's inverse system.
 
\begin{pro}
\label{soc2}
Let $H = (1,m,\Tt)$ be a numerical sequence. Then $H$ is the $h$-vector of an Artinian $2$-level local $k$-algebra if and only if $\Tt \ls \binom{m+1}{2}$.
\end{pro}
\begin{proof}[\scshape Proof]
We only have to prove that if $\Tt  \ls \binom{m+1}{2}$, then we are able to construct an Artinian $2$-level local algebra with $h$-vector $H$. By virtue of Proposition \ref{invsyscorr2} it is enough to exhibit a set $\underline{f}=\{f_1,\ldots,f_\Tt\}$ of polynomials of degree two such that the $h$-vector of $A_{\underline{f}}$ is $H$. Assume  $\Tt \gs m$ and set
\begin{equation*}
\g{T}_m^2 := {\rm Supp}\left((x_1,\ldots,x_m)^2\right) = \{x_1^2,x_1x_2,\ldots,x_m^2\}.
\end{equation*}
Consider $\Tt$ homogeneous polynomials $f_1,\ldots,f_\Tt$ corresponding to the first $\Tt$ terms of $\g{T}_m^2$ with respect to the lexicographic order:
$$
f_1 = x_1^2, \ \ f_2 = x_1x_2, \ \ f_3 = x_1x_3, \ \ \ldots
$$

Assume $\Tt < m$ instead, then consider $\Tt$ homogeneous polynomials $f_1,\ldots,f_\Tt$, defined as follows: 
\begin{equation*}
f_1 = x_1^2, \ \ f_2 =x_2^2, \ \ \ldots\ldots \ \ f_{\Tt-1} = x_{\Tt-1}^2, \ \  f_\Tt = x_\Tt^2 + x_{\Tt + 1}^2 + \ldots + x_m^2.
\end{equation*}
According to (\ref{HF}) it is easy to see that  in both cases $A_{\underline{f}}$ has $h$-vector $H = (1,m,\Tt)$.
\end{proof}
  
\subsection{Socle degree three}
It is well known that $h$-vectors of $s$-level algebras of embedding dimension $m$ and type $\Tt$ have to satisfy the following inequality (see for instance \cite{FRLK} and \cite{ChoI}):
$$
h_i \ls \min\left\{\Tt \binom{m+s-i-1}{s-i},\binom{m+i-1}{i}\right\} \ \ \ \mbox{ for } i=0,\ldots,s.
$$
If the socle degree is $s=3$ and the $h$-vector is $(h_0,h_1,h_2,h_3) = (1,m,n,\Tt)$ the only significant inequality is
$$
h_2 = n  \ls  \min\left\{\Tt m,\binom{m+1}{2}\right\}.
$$
Clearly $n \ls \binom{m+1}{2}$ by Macaulay's Theorem. Using Macaulay's inverse system and the Q-decomposition we can easily show that $n \ls \Tt m$.
\begin{pro}
\label{condnec}
Let $A=R/I$ be an Artinian level local $k$-algebra with $h$-vector $H = (1,m,n,\Tt)$. Then 
\begin{equation*}
n \ls \Tt m.
\end{equation*}
\end{pro}
\begin{proof}[\scshape Proof]
By Corollary \ref{corc1}
\begin{equation*}
n = HF_A(2) = HF_{Q(0)}(2).
\end{equation*}
Being $Q(0)$ a 3-level graded algebra there exist $F_1,\ldots,F_\Tt$ forms of degree three such that $Q(0) = A_{\underline{F}}$. Moreover, being $F_1,\ldots,F_\Tt$ homogeneous polynomials, it is easy to see that
\begin{equation*}
(\langle F_1,\ldots,F_\Tt\rangle_R)_2  =  \langle \partial \underline{F}\rangle_k  = \langle \partial F_1,\ldots,\partial F_\Tt\rangle_k.
\end{equation*}
Hence
\begin{equation*}
n =  HF_{A_{\underline{F}}}(2) = \dim_k (\langle \partial \underline{F}\rangle_k) \ls \Tt m
\end{equation*}
because $\partial F_1,\ldots,\partial F_\Tt$ are at most $\Tt m$ linearly independent forms of degree two. \\
\end{proof}

\vskip 3mm 
We prove now that the condition $n \ls \Tt m$, which is apparently a weak request, is sufficient to characterize all admissible $h$-vectors for level local $k$-algebras with socle degree three.  
\begin{teor}
\label{mainteor}
Let $H=(1,m,n,\Tt)$ be an $h$-vector. Then $H$ is the $h$-vector of an Artinian level local algebra if and only if $n \ls \Tt m$.
\end{teor}
The proof will be divided in three steps. In each case we will find a strategy for presenting  the generators of the inverse system (and hence the ideal). First we investigate the non-increasing vectors $m>n \gs \Tt >0$. Notice that in this range the $h$-vectors are achieved by level local algebras, not necessarily graded. 
\begin{pro}
\label{prop1}
Let $H = (1,m,n,\Tt)$ be an $h$-vector. If $m>n \gs \Tt >0$, then $H$ is the $h$-vector of an Artinian $3$-level local algebra.
\end{pro}
\begin{proof}[\scshape Proof] 
Consider the following $\Tt$ polynomials:
\begin{equation*}
f_1:=x_1^3, \ \ f_2:=x_2^3, \ \ f_3:=x_3^3, \ldots \ldots f_{\Tt-1}:=x_{\Tt-1}^3, \ \ f_\Tt:=x_\Tt^3+\ldots +x_n^3 + x_{n+1}^2 + \ldots + x_m^2.
\end{equation*}
By using (\ref{HF}) it is easy to prove that:
\begin{eqnarray*}
\left\{ \begin{array}{ll}
HF_{A_{\underline{f}}}(1) = \dim_k(\langle x_1,\ldots,x_m \rangle_k) = m \\ 
HF_{A_{\underline{f}}}(2) = \dim_k(\langle x_1^2,\ldots,x_n^2 \rangle_R) = n \\
HF_{A_{\underline{f}}}(3) = \Tt
\end{array} \right. \end{eqnarray*}
and hence $H$ is the $h$-vector of $A_{\underline{f}}$.
\end{proof}

In the next numerical range the $h$-vectors will be admissible for level {\it{graded}} algebras, not just local. This is because we are able to present homogeneous inverse system polynomials.
\begin{pro}
\label{proamm1}
Let $H = (1,m,n,\Tt)$ be an $h$-vector. If 
\begin{equation*}
\max \{\Tt,m\} \ls n \ls \Tt m
\end{equation*}
then $H$ is the $h$-vector of an Artinian $3$-level graded algebra. 
\end{pro}
\begin{proof}[\scshape Proof] 
For each $j=0,\ldots,\left[ \frac{m}{2} \right]$ we define:
\begin{equation*}
\call{D}_j := \{x_i^2x_{i+j}:i=1,\ldots m \},
\end{equation*}
where, by convention, whenever an index of a variable is $l \gs m+1$ one has to read $l-m$ (for example $x_{m+1}$ is $x_1$).
We equip each $\call{D}_j$ with the following order $\preceq_j$:
\begin{equation*}
x_i^2x_{i+j} \preceq_j x_k^2x_{k+j} \ \ \iff \ \  i \ls k. \ \ \ \ \ \ \ \ 
\end{equation*}
Denote by $\#A$ the cardinality of a finite set $A$. Observe that $\call{D}_i \cap \call{D}_j = \emptyset$ if $i\ne j$ and also $\# \call{D}_j = m$ for all $j$. Set $\call{D} = \cup_j \call{D}_j$. From the previous remarks we deduce that
\begin{equation*}
\# \call{D} = m \left( \left[ \frac{m}{2} \right] +1 \right) \gs \binom{m+1}{2}.
\end{equation*}
\underline{} \\
Using the terms of the set $\call{D}$ we want now to exhibit $\Tt$ forms of degree three $f_1,\ldots f_\Tt \in P=k[x_1,\ldots,x_m]$ with disjoint supports such that
\begin{equation*}
\# \mbox{Supp }\{f_1,\ldots,f_\Tt \} = \dim_k (\langle \partial f_1,\ldots, \partial f_\Tt \rangle_R) = n.
\end{equation*}
Since $H$ is an $h$-vector, we have:
\begin{equation*}
n=HF(2) \ls \binom{m+1}{2}.
\end{equation*}
Moreover $n \ls t m$ by assumption, therefore $n \ls \min \left \{ \binom{m+1}{2},\Tt m \right\}$. This shows that we can satisfy the request $\#$Supp $\{f_1,\ldots,f_\Tt \} = n$, since $n \ls \# \call{D}$. 

Set $\tilde{n} := n-\Tt$ ($n \gs \Tt$ by assumption), then there exist $h,l \in \g{N}$ such that
\begin{equation*}
\tilde{n} = h(m-1) + l \ \ \ \mbox{ with } 0 \ls l < m-1.
\end{equation*}
Notice that we have
\begin{equation}
\label{formula}
h(m-1) \ls h(m-1)+l = n-\Tt \ls \Tt m - \Tt = \Tt(m-1).
\end{equation}
If $m=1$, the $h$-vector is forced to be $(1,1,1,1)$ and therefore we have $\tilde{n} = h = 0 < 1 = \Tt$. If $m >1$ we can divide by $(m-1)$, hence we always get $\Tt \gs h$. Let us consider different cases: \\
\begin{compactenum}[\rm (a)]
\item If $\Tt=h$, then define
\begin{eqnarray*}
f_j := \sum_{i=1}^m x_i^2x_{i+j-1} \ \ \mbox{ for } j=1,\ldots,h.
\end{eqnarray*}
Notice that for all $j=1,\ldots,h,$ the monomials which appear in the support of $f_j$ are  elements of the set $\call{D}_{j-1}$. Also, being $\Tt = h$, the inequalities in (\ref{formula}) become equalities. Hence
\begin{eqnarray*}
\left \{
\begin{array}{ll}
n - \Tt = \Tt m- \Tt \\ 
h(m-1) = h(m-1) + l \end{array} \right.
\ \ \Rightarrow \ \
\left \{ \begin{array}{ll} 
n = \Tt m \\ 
l =0
\end{array} \right.
\end{eqnarray*}
Set now
\begin{equation*}
V^{(1)} :=  \langle \partial f_1,\ldots,\partial f_h \rangle_k = \langle x_ix_{i+j} : i = 1,\ldots,m \ \ j=0,\ldots,h-1 \rangle_k,
\end{equation*}
then one has
\begin{eqnarray*}
\left\{
\begin{array}{ll}
HF_{A_{\underline{f}}}(1) = \dim_k (\langle x_1,\ldots,x_m \rangle_k) = m \\
HF_{A_{\underline{f}}}(2) = \dim_k V^{(1)} = hm = t m = n \\
HF_{A_{\underline{f}}}(3) = h = \Tt
\end{array} \right.
\end{eqnarray*}
\item If $\Tt = h+1$, then
\begin{equation*}
n = n-\Tt + h+1 = \tilde{n} +h+1 = h(m-1) +l + h + 1 = hm+l+1.
\end{equation*}
Define
\begin{equation*}
f_{h+1} := \sum_{i=1}^{l+1} x_i^2 x_{i+h}
\end{equation*}
and consider now $\underline{f} = \{f_1,\ldots,f_h,f_{h+1}\}$, where $f_1,\ldots,f_h$ are defined as in (a). Set \\
\begin{equation*}
V^{(2)} :=  \langle \partial f_{h+1} \rangle_k,
\end{equation*}
then one has:
\begin{eqnarray*}
\left\{
\begin{array}{ll}
HF_{A_{\underline{f}}}(1) = \dim_k (\langle x_1,\ldots,x_m \rangle_k) = m \\
HF_{A_{\underline{f}}}(2) = \dim_k (V^{(1)} \oplus V^{(2)}) = (hm) + (l+1) = n \\
HF_{A_{\underline{f}}}(3) = h+1 = \Tt
\end{array} \right.
\end{eqnarray*}
\underline{}\\
\item If $h+2 \ls \Tt \ls h+m-l$, then we enlarge the set of the polynomials already defined in (a) and (b) as follows:
\begin{eqnarray*}
\left\{ \begin{array}{ll}
f_{h+2} := x_{l+2}^2x_{l+h+2} \\
\ldots \\
f_\Tt := x_{l+\Tt-h}^2x_{l+\Tt}
\end{array} \right. 
\end{eqnarray*}
We set \\
\begin{equation*}
V^{(3)}_1 := \langle \partial f_{h+2},\ldots, \partial f_\Tt \rangle_k = \langle x_ix_{i+h} : i=l+2,\ldots,l+\Tt-h \rangle_k.
\end{equation*}
We get 
\begin{eqnarray*}
\left\{
\begin{array}{ll}
HF_{A_{\underline{f}}}(1) = \dim_k (\langle x_1,\ldots,x_m \rangle_k) = m \\
HF_{A_{\underline{f}}}(2) = \dim_k (V^{(1)} \oplus V^{(2)} \oplus V^{(3)}_1) = (hm) + (l+1) + (\Tt-h-1)= \\
\ \ \ \quad \qquad \qquad \ \ \qquad \ \qquad \qquad \ \qquad \qquad = h(m-1)+l + \Tt = \tilde{n} +\Tt = n \\
HF_{A_{\underline{f}}}(3) = \Tt
\end{array} \right.
\end{eqnarray*}
\underline{}\\
\item Let $\Tt > h+m-l$, and consider the set of polynomials $\{f_1,\ldots,f_{h+m-l}\}$ already defined in (a),(b), taking $\Tt$ maximum (equal to $h+m-l$) in the range of (c). Then we have:
\begin{eqnarray*}
\left\{ \begin{array}{ll}
f_{h+2} := x_{l+2}^2x_{l+h+2} \\
\ldots \\
f_{h+m-l} := x_m^2x_h
\end{array} \right. 
\end{eqnarray*}
which are polynomials that involve all the terms of $\call{D}_h$ not yet considered in $f_{h+1}$. Our strategy is to define $f_{h+m-l+1},\ldots,f_\Tt$ each one as an element of $\call{D}_j$, with $j >h$, following the orders $\preceq_j$. In particular, starting from $\call{D}_{h+1}$, we pick all the elements in this set before passing to $\call{D}_{h+2}$, and so on. Set
\begin{equation*}
\overline{n} := n-(h+1)m.
\end{equation*}
Note that
\begin{equation*}
\overline{n} = h(m-1)+l + \Tt - (h+1)m = \Tt - (h+m-l) >0
\end{equation*}
and so we can write
\begin{equation*}
\overline{n} = mr+s,
\end{equation*}
where $r,s \in \g{N}$ and $0 \ls s <m$. In this way we get:
\begin{eqnarray*}
\left\{ \begin{array}{ll} \left. \begin{array}{ll} 
f_{h+2} := x_{l+2}^2x_{l+h+2} \\
\ldots \\
f_{h+m-l} := x_m^2x_h \end{array} \right] \mbox{ we complete $\call{D}_h$ } \\ 
\left. \begin{array}{ll} 
f_{h+m-l+1} := x_1^2x_{h+2} \\
\ldots \\
\ldots \\
\end{array} \right] \mbox{ we complete $\call{D}_{h+1}$ } \\
\ldots \\
\ldots \\
\left. \begin{array}{ll}
\ldots \\
\ldots \\
f_{\Tt - s} := x_m^2x_{h+r} \end{array} \right] \mbox{ we complete $\call{D}_{h+r}$ } 
\end{array} \right. 
\end{eqnarray*}
With these choices we define
\begin{equation*}
V^{(3)}_2 := \langle \partial f_{h+2},\ldots, \partial f_{\Tt-s} \rangle_k = \langle x_ix_{i+j} : i=1,\ldots,m \ j=h,\ldots,h+r \rangle_k.
\end{equation*}
Setting $\underline{f} = \{f_1,\ldots,f_{h+1},f_{h+2},\ldots,f_{\Tt-s}\}$ we get
\begin{eqnarray*}
\left\{
\begin{array}{ll}
HF_{A_{\underline{f}}}(1) = \dim_k (\langle x_1,\ldots,x_m \rangle_k) = m \\
HF_{A_{\underline{f}}}(2) = \dim_k (V^{(1)} \oplus V^{(2)} \oplus V^{(3)}_2) = (hm) + (l+1) + (m-(l+1)+rm) = \\
\ \ \ \quad \qquad \qquad \ \ \qquad \ \qquad \qquad \ \qquad \qquad = m(h+1) + \overline{n}-s = n-s \\
HF_{A_{\underline{f}}}(3) = \Tt - s
\end{array} \right.
\end{eqnarray*}
If $s=0$ we have the required $\Tt$ polynomials. If $s >0$ just define the last $s$ polynomials following the above strategy, that is setting each one of them as a term of $\call{D}_{h+r+1}$ (respecting the order $\preceq_{h+r+1}$):
\begin{eqnarray*}
\left\{ \begin{array}{ll}
f_{\Tt-s+1} := x_1^2x_{h+r+2} \\
\ldots \\
f_\Tt := x_s^2x_{h+r+s+1}
\end{array} \right. 
\end{eqnarray*}
In such a way we can define
\begin{equation*}
V^{(4)} := \langle \partial f_{\Tt-s+1},\ldots, \partial f_\Tt \rangle_k = \langle x_ix_{i+h+r+1} : i=1,\ldots,s  \rangle_k
\end{equation*}
and we get
\begin{eqnarray*}
\left\{
\begin{array}{ll}
HF_{A_{\underline{f}}}(1) = \dim_k (\langle x_1,\ldots,x_m  \rangle_k) = m \\
HF_{A_{\underline{f}}}(2) = \dim_k (V^{(1)} \oplus V^{(2)} \oplus V^{(3)}_2 \oplus V^{(4)}) = n-s + s = n \\
HF_{A_{\underline{f}}}(3) = \Tt
\end{array} \right.
\end{eqnarray*}
\end{compactenum}
Once again $A_{\underline{f}}$ has $h$-vector $H$.
\end{proof}

\begin{pro}
\label{proamm3}
Let $H = (1,m,n,\Tt)$ be an $h$-vector. If $n < \Tt$ then $H$ is the $h$-vector of an Artinian $3$-level local algebra.
\end{pro}
\begin{proof}[\scshape Proof]
Since $H$ is an $h$-vector, it has to satisfy Macaulay's Theorem (see \cite{Stan} for notation). Therefore
\begin{equation*}
\Tt = HF_A(3) \ls HF_A(2)^{\langle2\rangle} = n^{\langle2\rangle}.
\end{equation*}
There exists $l \in \g{N}$, $l \gs 1$ and $h \in\{0,1,2,\ldots,l\}$  such that
\begin{equation*}
n = \binom{l+1}{2} + h,
\end{equation*} 
hence
\begin{equation*}
n^{\langle2\rangle} = \binom{l+2}{3} + \binom{h+1}{2}.
\end{equation*}
Moreover, again for the fact that $H$ is an $h$-vector, it has to be
\begin{equation*}
\binom{l+1}{2} \ls \binom{l+1}{2}+h =n \ls m^{\langle1\rangle} = \binom{m+1}{2}
\end{equation*} 
and so $m \gs l$ and $m=l$ only when $h=0$. 

Let us extend the notation introduced in the proof of Proposition \ref{soc2}:
\begin{equation*}
\g{T}_l^n := \mbox{ Supp} \left( (x_1+\ldots+x_l)^n \right) = \{ x_1^n,x_1^{n-1}x_2,\ldots,x_l^n \}.
\end{equation*}
Notice that $\# \g{T}_l^n = \binom{l+n-1}{n}$. For every $f \in R$ we also introduce the following notation:
\begin{equation*}
f \cdot \g{T}_l^n := \{f t : t \in \g{T}_l^n \}.
\end{equation*}
\underline{}\\
Assume now $h=0$, then $\Tt \ls \binom{l+2}{3}$ because $H$ is an $h$-vector. It is enough to define $\Tt$ forms of degree three $F_1,\ldots,F_\Tt$ as the first $\Tt$ monomials in $\g{T}_l^3$ with respect to Lexicographic order. Notice that $\# \g{T}_l^3 = \binom{l+2}{3}$ and, as said before $\Tt \ls \binom{l+2}{3}$, so this choice is possible. In this way we are selecting $\Tt$ distinct monomials of degree three which generate $n=\binom{l+1}{2}$ linearly independent forms of degree two, in other words we have \\
\begin{eqnarray*}
\left\{
\begin{array}{ll}
HF_{A_{\underline{F}}}(1) = \dim_k ( \langle x_1,\ldots,x_l \rangle_k) = k \\
HF_{A_{\underline{F}}}(2) = \dim_k \g{T}_l^2 = \binom{l+1}{2} = n \\
HF_{A_{\underline{F}}}(3) \dim_k(\langle F_1,\ldots,F_\Tt\rangle_k) = \Tt
\end{array} \right.
\end{eqnarray*}
Being $HF_A(1) = m \gs l$, we can set 
\begin{equation*}
f_1 := F_1 + \sum_{i=l+1}^m x_{i}^2, \quad \qquad  f_j := F_j \ \  \mbox{ for all } j=2,\ldots,\Tt
\end{equation*}
and with these choices $A_{\underline{f}}$ has $h$-vector $\left(1,m, \binom{l+1}{2},\Tt \right) = (1,m,n,\Tt)$, as required. 

Now assume $n=\binom{l+1}{2}+h$ with $0 < h \ls l$ and recall
\begin{equation*}
\# \g{T}_l^2 = \binom{l+1}{2}.
\end{equation*}
In this case we consider $n = \binom{l+1}{2} + h$ forms of degree three $F_1,\ldots,F_n$ as follows. Each of the first $\binom{l+1}{2}$ is selected to be a different monomial in the set $x_1 \cdot \g{T}_l^2 = \{ x_1^3,x_1^2x_2,\ldots,x_1x_l^2\}$ (recall that $\# (x_1 \cdot \g{T}_l^2) = \binom{l+1}{2}$), once again following lexicographic order:
\begin{equation*}
F_1 := x_1^3, \ \ \ F_2 := x_1^2x_2,\ \ \ \ldots\ldots\ \ \ F_{\binom{l+1}{2}} := x_1x_l^2.
\end{equation*}
We define the remaining $h$ in the following way
\begin{equation*}
F_{\binom{l+1}{2} + 1} := x_{l+1}^3, \ \ \ F_{\binom{l+1}{2} + 2} := x_1x_{l+1}^2, \ \ \ \ldots\ldots\ \ \ F_n :=x_{h-1}x_{l+1}^2.
\end{equation*}
These are $n = \binom{l+1}{2}+h$ linearly independent forms of degree three whose derivatives generate a $k$-vector space of dimension exactly $n = \binom{l+1}{2} + h$. In fact:
\begin{eqnarray*}
\dim_k (\langle \partial F_1,\ldots,\partial F_n \rangle_k) = \qquad \qquad \qquad \qquad \qquad\\
= \dim_k(\langle \g{T}_l^2 \rangle_k \oplus \langle x_{l+1}^2, x_1 x_{l+1}, \ldots,x_{h-1}x_{l+1} \rangle_k ) = \binom{l+1}{2} + h = n.
\end{eqnarray*}
Set $\underline{F} = \{F_1,\ldots,F_n\}$, then one has
\begin{eqnarray*}
\left\{
\begin{array}{ll}
HF_{A_{\underline{F}}}(1) = \dim_k (\langle x_1,\ldots,x_{l+1} \rangle_k) = l+1 \\
HF_{A_{\underline{F}}}(2) = \dim_k (\langle \g{T}_l^2 \rangle \oplus \langle x_{l+1}^2, x_1 x_{l+1}, \ldots,x_{h-1}x_{l+1} \rangle_k) = n \\
HF_{A_{\underline{F}}}(3) = \dim_k  (\langle \g{T}_l^3 \rangle_k \oplus \langle x_{l+1}^3,x_1x_{l+1}^2,\ldots,x_{h-1}x_{l+1}^2 \rangle_k) = n
\end{array} \right.
\end{eqnarray*}
Recall that by assumption and by Macaulay's Theorem we have
\begin{equation*}
n < HF_A(3) = \Tt \ls \binom{l+2}{3} + \binom{h+1}{2}.
\end{equation*}
We consider now $\Tt - n$ new polynomials. Notice that
\begin{equation*}
\Tt - n \ls \binom{l+2}{3} + \binom{h+1}{2} - n = \binom{l+1}{3} + \binom{h}{2} = \# \left[ \left(\g{T}_l^3 \smallsetminus \left(x_1 \cdot \g{T}_l^2 \right)\right) \cup \left( x_{l+1} \cdot \g{T}_{h-1}^2 \right)\right].
\end{equation*}
Hence we can define $F_{n+1},\ldots,F_\Tt$ each one as a monomial in $\left[ \left(\g{T}_l^3 \smallsetminus \left(x_1 \cdot \g{T}_l^2 \right)\right) \cup \left( x_{l+1} \cdot \g{T}_{h-1}^2 \right)\right]$, chosen with respect to any order (even randomly). The introduction of these new $\Tt -n$ forms does not modify the value of $HF_A(2)$, in fact
\begin{equation*}
\langle \partial F_{n+1},\ldots,\partial F_\Tt \rangle_k \subseteq \langle \g{T}_l^2  \rangle \oplus \langle x_{l+1}^2, x_1 x_{l+1}, \ldots,x_{h-1}x_{l+1} \rangle_k = \langle  \partial F_1,\ldots,\partial F_n  \rangle_k.
\end{equation*}
As above, being $HF_A(1) = m \gs l+1$, set:
\begin{equation*}
f_1 := F_1 + \sum_{i=l+2}^m x_{i}^2 \quad \qquad f_j := F_j \ \ \mbox{ for all } j=2,\ldots,\Tt
\end{equation*}
With this choices $A_{\underline{f}}$ has $h$-vector $H$.
\end{proof}
\noindent
Propositions \ref{prop1}, \ref{proamm1}, \ref{proamm3} complete the proof of Theorem \ref{mainteor}. 

\medskip

If  we restrict the investigation to codimension $m = HF_A(1) = 3$, following our strategy,  we construct graded level algebras in all cases except for the $h$-vectors:
\begin{equation*}
(1,3,2,1) \qquad \qquad \qquad (1,3,2,2) \qquad \qquad \qquad (1,3,3,4).
\end{equation*}
A priori this fact does not mean that these cannot be admissible for a graded level algebra, but actually this is the case. The first one is clearly not admissible as the $h$-vector of a graded Gorenstein algebra, because it is not symmetric. As regarding the other two, a reason comes from a result proved by Peeva \cite{Peeva} on the consecutive cancellations in a free resolution of the associated  lex segment ideal. See also \cite{GeramitaMigliore}. 

Our strategy does not give a characterization of the $h$-vectors admissible for level graded algebras. In fact the  condition  $n \gs \max\{\Tt,m\}$ of Proposition \ref{proamm1}   is sufficient, but not necessary if the embedding dimension is $m \gs 4 $ as the following example shows.

\medskip 
 
\begin{ese}
\label{es1}
Consider $H = (1,4,5,6)$. In this case $n = 5 < t = 6$ hence, following the construction of Proposition \ref{proamm3}, we exhibit a non-graded example.  
However, if we consider  the following homogeneous polynomials:
\begin{equation*}
g_1  = x_1^3, \ \ \ g_2 :=x_1^2x_2, \ \ \ g_3 :=x_1x_2^2, \ \ \ g_4 :=x_2^3, \ \ \ g_5 :=x_3^3, \ \ \ g_6 :=x_4^3
\end{equation*}
the $h$-vector of $A_{\underline{g}}$ is  $H = (1,4,5,6)$ and $A_{\underline{g}}$ is a graded level algebra.
\end{ese}
\medskip 
If $n < \min \{\Tt, m\}, $  the $h$-vector is non-unimodal and the construction of a   graded level algebra is in general not trivial.

\begin{ese}
\label{es2}
Consider $(1,13,12,13)$, we want to prove that it can be the $h$-vector of a level graded algebra even if, following our strategy, we end up with a non-homogeneous ideal for this $h$-vector. Set
\begin{equation*}
F :=a_1x^3+a_2x^2y+a_3x^2z+a_4xy^2+a_5xyz+a_6xz^2+a_7y^3+a_8y^2z+a_9yz^2 + a_{10}z^3.
\end{equation*}
Then $A_F$ has non-unimodal $h$-vector $(1,13,12,13,1)$. Hence $A_{\underline{\partial{F}}}$, which is obtained by truncation from $A_F$, has $h$-vector $(1,13,12,13)$. This non-unimodal Gorenstein example was first given by Stanley. 

\end{ese}

Also if $m>n \gs \Tt$ it is not immediate to construct graded examples, but it is possible to get one starting from Example \ref{es2}.
\begin{ese}
\label{es3}
Consider $H = (1,15,14,14)$. It satisfies the condition of Theorem \ref{mainteor}, hence it is admissible for a level local $k$-algebra. Indeed we construct, in Proposition \ref{prop1}, non homogeneous inverse system polynomials for this $h$-vector. However $H$ is admissible also for a graded level algebra. Set:
\begin{equation*}
F :=a_1x^3+a_2x^2y+a_3x^2z+a_4xy^2+a_5xyz+a_6xz^2+a_7y^3+a_8y^2z+a_9yz^2 + a_{10}z^3,
\end{equation*}
then the truncation $A_{\underline{\partial{F}}}$ of $A_F$ has $h$-vector $(1,13,12,13)$, as shown in Example \ref{es2}. Consider
\begin{equation*}
G:=y_1^3 + y_2^3,
\end{equation*}
then $A_G$ has $h$-vector $(1,2,2,1)$. If we set $\underline{E}:=\{ \underline{\partial{F}},G\}$ ,then $A_{\underline{E}}$ is graded and has $h$-vector:
\begin{equation*}
(1,13,12,13) + (1,2,2,1) = (1,15,14,14).
\end{equation*}
\end{ese}
\vskip 5mm
{\bf Acknowledgements.} We warmly thank Professor M. E. Rossi for her comments and suggestions which improved the exposition of this article. We also wish to thank the anonymous referee for his careful review.

\end{document}